\documentclass{amsart}
\pdfoutput=1

\usepackage{microtype}
\usepackage{booktabs}
\usepackage{mathtools}
\usepackage[colorlinks,citecolor=black,linkcolor=black,urlcolor=black,pdftitle={Improved kissing numbers in seventeen through twenty-one dimensions},pdfauthor={Henry Cohn and Anqi Li}]{hyperref}
\newcommand{\arXiv}[1]{arXiv:\href{https://arxiv.org/abs/#1}{#1}}
\usepackage{doi}

\newtheorem{theorem}{Theorem}[section]

\newtheorem{lemma}[theorem]{Lemma}

\numberwithin{equation}{section}
\numberwithin{table}{section}

\newcommand{\F}{\mathbb{F}}
\newcommand{\R}{\mathbb{R}}
\newcommand{\Z}{\mathbb{Z}}
\newcommand{\RM}{\operatorname{RM}}

\title[Improved kissing numbers in $17$ through $21$ dimensions]{Improved kissing numbers in seventeen\\ through twenty-one dimensions}

\author{Henry Cohn}
\address{Microsoft Research New England\\
One Memorial Drive\\
Cambridge, MA 02142}
\curraddr{Department of Mathematics\\
Massachusetts Institute of Technology\\
Cambridge, MA 02139}
\email{cohn@mit.edu}

\author{Anqi Li}
\address{Department of Mathematics\\
Stanford University\\
Stanford, CA 94305}
\curraddr{Department of Computer Science\\
University of Washington\\
Seattle, WA 98195}
\email{aqli@cs.washington.edu}

\begin{document}

\begin{abstract}
We prove that the kissing numbers in $17$, $18$, $19$, $20$, and $21$ dimensions are at least $5730$, $7654$, $11692$, $19448$, and $29768$, respectively. The previous records were set by Leech in 1967, and we improve on them by $384$, $256$, $1024$, $2048$, and $2048$. Unlike the previous constructions, the new configurations are not cross sections of the Leech lattice minimal vectors. Instead, they are constructed by modifying the signs in the lattice vectors to open up more space for additional spheres. 
\end{abstract}

\maketitle

\section{Introduction}

The \emph{kissing problem} in $\R^n$ asks how many unit spheres can be
arranged tangent to a central unit sphere so that their interiors are
disjoint; it can be viewed as a packing problem for spherical caps or
as a local analogue of the sphere packing problem in Euclidean space. One can
also formulate the kissing problem in terms of the points of tangency: how
large a subset $\mathcal{C}$ of the unit sphere $S^{n-1}$ is there such that
$\langle x,y \rangle \le 1/2$ for all $x,y \in \mathcal{C}$ with $x \ne y$?
Here, $\langle \cdot,\cdot\rangle$ denotes the usual inner product on $\R^n$.
In this formulation, which we will use in this paper, the goal is a
\emph{spherical code} with minimal angle at least $\pi/3 = \arccos(1/2)$.

The exact kissing number is known for $n \in \{1, 2, 3, 4, 8, 24\}$, but no other
cases; see Table~\ref{table:kissing} for the best lower and upper bounds
currently known. For $n \le 8$, optimal kissing configurations appear to
involve familiar mathematical structures: the $A_1$, $A_2$, $A_3$ (or,
equivalently, $D_3$), $D_4$, $D_5$, $E_6$, $E_7$, and $E_8$ root systems are
optimal or conjectured to be optimal, although they are not always unique
(see \cite{L67b,L69,CJKT11,Sz23,CR26} for five through seven dimensions).
For $8 < n < 16$, there have been a number of diverse constructions, most
recently by Ganzhinov \cite{G22}, AI models \cite{DeepMind25,BKPZ26},
and Takhanov, Assylbekov, and Yun \cite{TAY26}. For $n=16$ the answer may be
the minimal vectors in the Barnes-Wall lattice, and for $n=24$, the answer is
the minimal vectors in the Leech lattice. However, for $16 < n < 24$, before
the present paper there had been no improvements on Leech's constructions
\cite{L67} from 1967. These configurations were all obtained as cross sections
of the $24$-dimensional kissing configuration; they can also be described as
the minimal vectors in the laminated lattices \cite{CS80}.

\begin{table}
\caption{Lower and upper bounds for kissing numbers, including this paper's
improved lower bounds in dimensions $17$ through $21$ and Ho's subsequent
improvement in dimension $19$. See \cite{C} for records in up to $48$ dimensions.}
\label{table:kissing}
\centering
\begin{tabular}{rrrr}
\toprule
Dimension & Lower bound & Upper bound & References\\
\midrule
$1$ & $2$ & $2$ & \\
$2$ & $6$ & $6$ & \\
$3$ & $12$ & $12$ & \cite{SW53} \\
$4$ & $24$ & $24$ & \cite{S01,M08} \\
$5$ & $40$ & $44$ & \cite{KZ73,MV10} \\
$6$ & $72$ & $77$ & \cite{KZ73,dLLdMK24} \\
$7$ & $126$ & $134$ & \cite{KZ73,MV10} \\
$8$ & $240$ & $240$ & \cite{KZ73,L79,OS79} \\
$9$ & $306$ & $363$ & \cite{LS71,MO18} \\
$10$ & $510$ & $553$ & \cite{G22, MO18} \\
$11$ & $604$ & $868$ & \cite{BKPZ26,dLL22} \\
$12$ & $841$ & $1355$ & \cite{TAY26,dLL22} \\
$13$ & $1154$ & $2064$ & \cite{ZE99,dLL22} \\
$14$ & $1932$ & $3174$ & \cite{G22,dLL22} \\
$15$ & $2564$ & $4853$ & \cite{LS71,dLL22} \\
$16$ & $4320$ & $7320$ & \cite{BW59,dLL22} \\
$17$ & $5730$ & $10978$ & \cite{dLL22} \\
$18$ & $7654$ & $16406$ & \cite{dLL22} \\
$19$ & $11948$ & $24417$ & \cite{H26,dLL22} \\
$20$ & $19448$ & $36195$ & \cite{dLL22} \\
$21$ & $29768$ & $53524$ & \cite{dLL22} \\
$22$ & $49896$ & $80810$ & \cite{L67,dLL22} \\
$23$ & $93150$ & $122351$ & \cite{L67,dLL22} \\
$24$ & $196560$ & $196560$ & \cite{L67,L79,OS79} \\
\bottomrule
\end{tabular}
\end{table}

For $17 \le n \le 21$, we obtain improvements on these constructions.
Specifically, we improve Leech's lower bounds (which were $5346$, $7398$,
$10668$, $17400$, and $27720$) by $384$, $256$, $1024$, $2048$, and $2048$:

\begin{theorem} \label{thm:main}
The kissing numbers in $\R^{17}$, $\R^{18}$, $\R^{19}$, $\R^{20}$, and $\R^{21}$ are at least
$5730$, $7654$, $11692$, $19448$, and $29768$, respectively.
\end{theorem}

Ho \cite{H26} has since improved the lower bound in dimension $19$ to $11948$ by
combining our techniques with an improved choice of error-correcting code.

Our constructions cannot be obtained as cross sections of the Leech lattice
minimal vectors (via a vector subspace of $\R^{24}$), as we will see from their inner products.
We cannot rule out the possibility that there might
be other cross sections larger than Leech's or even our constructions, but we expect that Leech's
cross sections are as large as possible.

To prove Theorem~\ref{thm:main}, we modify the signs in the known constructions to produce space for additional spheres. It is plausible that this approach could work in other dimensions, with $22$ dimensions being a reasonable target, but we have not had any luck in that case.

The proof of Theorem~\ref{thm:main} does not depend on any computer calculations. However, we have made available files with explicit coordinates for our kissing configurations as part of the supplementary data \cite{CL} for our paper. We also supply code to verify these files algorithmically, which provides an additional check that we have not overlooked anything in our proofs.

In the next section, we describe the simplest case, namely dimensions $19$ through $21$. We then proceed in Sections~\ref{sec:17} and~\ref{sec:18} to the modifications required in $17$ and $18$ dimensions. Finally, we conclude with a discussion of potential generalizations in Section~\ref{sec:gen}.

\section{Dimensions 19, 20, and 21}
\label{sec:19-21}

We begin with a relatively simple construction of the previously known kissing configurations in dimensions $19$ through $21$. Suppose $C$ is a binary error-correcting code of block length~$n$, constant weight~$8$, and minimal distance~$8$. Then there is a kissing configuration in $\R^n$ of size $2^2 \binom{n}{2} + 2^7 |C|$, which consists of the $2^2 \binom{n}{2}$ permutations of $(\pm 2, \pm 2, 0, \dots, 0)$ and the $2^7 |C|$ vectors that put an even number of minus signs on the codewords in $C$. Each of these vectors has squared norm~$8$, and it is straightforward to check that the inner products between distinct vectors are at most $4$ (different codewords in $C$ overlap in at most four coordinates, and different sign patterns for the same codeword differ in at least two signs).

To obtain a suitable code $C$, we start with the $S(5,8,24)$ Steiner system, which achieves $|C|=759$ when $n=24$. We then repeatedly shorten it to include only codewords with $0$ in a specified coordinate (chosen greedily to produce large results), after which we drop that coordinate. The resulting sizes are shown in Table~\ref{table:C}; note that these sizes depend on the choice of deleted coordinates beyond five deletions. This process recovers the largest code $C$ known for each block length $n$ except $n=17$, for which $|C|=34$ can be achieved (see \cite{B,BSSS90} for a compilation of records).

\begin{table}
\caption{The sizes of the shortened $S(5,8,24)$ Steiner system and the resulting kissing configuration in $n$ dimensions.}
\label{table:C}
\centering
\begin{tabular}{rrr}
\toprule
$n$ & $|C|$ & $2^2 \binom{n}{2} + 2^7 |C|$\\
\midrule
$16$ & $30$ & $4320$\\
$17$ & $30$ & $4384$\\
$18$ & $46$ & $6500$\\
$19$ & $78$ & $10668$\\
$20$ & $130$ & $17400$\\
$21$ & $210$ & $27720$\\
$22$ & $330$ & $43164$\\
$23$ & $506$ & $65780$\\
$24$ & $759$ & $98256$\\
\bottomrule
\end{tabular}
\end{table}

The lower bounds $2^2 \binom{n}{2} + 2^7 |C|$ in Table~\ref{table:C} are generally worse than the previously known records, but they agree for $n=16$, $19$, $20$, and $21$. One can check that the construction described above is isometric to the minimal vectors of the laminated lattice $\Lambda_n$ in these cases.

The initial impetus for the present paper was the question of what happens if we use an odd number of minus signs in this construction, instead of an even number. This variation still produces a kissing configuration of the same size, for exactly the same reason. At first it sounds like a superficial change, and one might wonder whether the result is simply isometric to the configuration with an even number of minus signs. However, it is not. (Suppose $x$ and $y$ are in the configuration and $|x+y|^2=8$. In the even case, $x+y$ must also be in the configuration, but that is not true in the odd case.) Instead, it is sufficiently different that it opens up space we can use to improve the kissing configurations with $19 \le n \le 21$.

To see why, consider the following construction over the field $\F_2$ with two elements. Suppose $c \in \F_2^n$ is orthogonal to the codewords in $C$ modulo~$2$ (for example, it could be the zero vector), and define $v \in \R^n$ by $v_i = (-1)^{c_i}\sqrt{8/n}$, so that $|v|^2=8$, which is the same normalization as the kissing configuration. We can enlarge the kissing configuration to include $v$ if $\langle v,w \rangle \le 4$ for all vectors $w$ already in the kissing configuration. When $w$ is a permutation of $(\pm 2, \pm 2, 0, \dots, 0)$, that inequality holds because $n\ge 8$. When $w$ is a signed vector with support in $C$, the vector $v$ must have an even number of minus signs in the support of $w$, because $c$ is orthogonal to the code $C$ modulo~$2$. Because $w$ has an odd number of minus signs, their signs must differ in at least one coordinate, and their inner product $\langle v,w \rangle$ is therefore at most
\[
7 \sqrt{\frac{8}{n}} - \sqrt{\frac{8}{n}} = 6 \sqrt{\frac{8}{n}},
\]
which is at most $4$ when $n \ge 18$.

Taking $c=0$ already improves the previous kissing number bounds in $19$ through $21$ dimensions by one. (It fails to improve on $18$, $22$, or $23$ dimensions because the starting value from Table~\ref{table:C} is too low.) Note that this construction fails if we use an even number of minus signs in the initial kissing configuration, since no choice of $c$ can guarantee a sign difference between $v$ and $w$ in that case.

To prove Theorem~\ref{thm:main} for $n \ge 19$, we simply include many of these vectors $v$. Suppose $c$ and $c'$ in $\F_2^n$ lead to vectors $v$ and $v'$ of squared norm~$8$ as above, and the Hamming distance between $c$ and $c'$ is $d$. Then 
\[
\langle v,v' \rangle = (n-d) \cdot \frac{8}{n} - d \cdot \frac{8}{n},
\]
which is at most $4$ when $d \ge n/4$. In other words, we can obtain additional vectors from any subcode of the orthogonal complement of $C$ that has minimal distance at least $n/4$.

To find such a code, we begin with the extended binary Golay code, a linear code of block length~$24$ and dimension~$12$. It is self-dual, and its weight~$8$ codewords are the $S(5,8,24)$ Steiner system. Thus, there are $2^{12}=4096$ codewords orthogonal to the Steiner system when $n=24$, with minimal distance~$8$ between them. Puncturing this code (i.e., deleting a coordinate) twice gives $4096$ codewords with minimal distance~$6$ when $n=22$. Now shortening the code (i.e., restricting to codewords with zero in the last coordinate and then deleting that coordinate) gives $2048$ codewords with minimal distance~$6$ when $n=21$. Puncturing gives $2048$ with minimal distance~$5$ when $n=20$, and shortening that gives $1024$ with minimal distance~$5$ when $n=19$. Throughout this process, we delete the same coordinates as in the construction of $C$, so that the resulting codewords are all orthogonal to $C$. These minimal distances are all at least $n/4$, and so we obtain the improvements from Theorem~\ref{thm:main}.

The resulting kissing configurations cannot be cross sections of the Leech lattice minimal vectors, because they have inner products that cannot occur in such a cross section. For example, $4\sqrt{8/n}$ is such an inner product. If we normalize the vectors to be unit vectors, this inner product becomes $\sqrt{2/n}$, which is not within the set $\{\pm 1, \pm 1/2, \pm 1/4, 0\}$ of normalized inner products among the Leech lattice minimal vectors. More generally, the inner products in our $n$-dimensional kissing configuration with $19 \le n \le 21$ consist of the inner products $\pm 1, \pm 1/2, \pm 1/4, 0$ from the Leech lattice, the irrational inner products given by $\sqrt{2/n}$ times $\pm1/2, \pm 1, \pm 3/2$, and the inner products $1-2k/n$ with $d \le k \le 16$, where $d$ is the minimal distance of the code from the previous paragraph (except $k=9$ and $k=13$ are omitted when $n=21$, because those distances do not occur in the code).

To produce improvements when $n=17$ or $18$, we must take a somewhat different approach. For those cases, we begin with the $16$-dimensional kissing configuration that uses an odd number of minus signs, which we call the \emph{odd kissing configuration}; for comparison, the even kissing configuration is the original Barnes-Wall kissing configuration. We do not know how to use the odd kissing configuration to obtain an improvement in $16$ dimensions, but we can extend it to larger configurations in $17$ and $18$ dimensions. It was previously known that there was more than one $4320$-point kissing configuration in $16$ dimensions (see, for example, \cite[p.~2268]{CJKT11}), but as far as we know, these variants had not proved useful for other constructions.

\section{Dimension 17}
\label{sec:17}

To obtain an improvement in $17$ dimensions, we begin by examining $16$ dimensions. First, we will carefully study the binary codes that arise in constructing the odd $16$-dimensional kissing configuration.  In doing so, we will identify $\F_2^{16}$ with the space $\F_2^{4 \times 4}$ of $4 \times 4$ binary matrices. 

Let $C_{10} \subseteq \F_2^{4 \times 4}$ be the code consisting of matrices whose rows and columns all have the same sum in $\F_2$ (the same for both rows and columns, not rows and columns separately), and let $C_6 := C_{10}^{\perp}$ be its dual code. Let $G$ be the automorphism group of $C_{10}$ and $C_6$ (i.e., the group of coordinate permutations preserving either and hence both of these codes), and let $G_0$ be the subgroup of $G$ generated by permutations of the rows, permutations of the columns, and the transpose map, with $|G_0| = 2 \cdot 4!^2 = 1152$.

Note that $C_{10}$ is a linear code of dimension~$10$, since there are only six linearly independent constraints: the four row sums automatically have the same total as the four column sums. The dual code $C_6 = C_{10}^\perp$ has size $64$ and contains $\binom{4}{2}^2/2 = 18$ matrices of the form
\[
\begin{bmatrix}
1 & 1 & 0 & 0\\
1 & 1 & 0 & 0\\
0 & 0 & 1 & 1\\
0 & 0 & 1 & 1
\end{bmatrix}
\]
(i.e., equivalent to this matrix under the action of $G_0$),
$2 \binom{4}{2} = 12$ of the form
\[
\begin{bmatrix}
1 & 1 & 1 & 1\\
1 & 1 & 1 & 1\\
0 & 0 & 0 & 0\\
0 & 0 & 0 & 0
\end{bmatrix},
\]
$4^2 = 16$ of the form
\[
\begin{bmatrix}
0 & 1 & 1 & 1\\
1 & 0 & 0 & 0\\
1 & 0 & 0 & 0\\
1 & 0 & 0 & 0
\end{bmatrix},
\]
$4^2 = 16$ of the form
\[
\begin{bmatrix}
1 & 0 & 0 & 0\\
0 & 1 & 1 & 1\\
0 & 1 & 1 & 1\\
0 & 1 & 1 & 1
\end{bmatrix},
\]
and also the zero matrix and the all $1$'s matrix. Note that $C_6 \subseteq C_{10}$. We will refer to these four forms as \emph{squares}, \emph{pairs}, \emph{crosses}, and \emph{anticrosses}. One can check that the crosses generate $C_6$ as an $\F_2$-vector space, as do the anticrosses, while the squares and pairs together do not.

Each element of $G_0$ is an automorphism of $C_{10}$ and $C_6$, but $G_0$ is not the full automorphism group $G$. Instead, $G$ is a group of order $10|G_0| = 11520$, which is generated by $G_0$ and the involution
\[
\begin{bmatrix}
x_0 & x_1 & x_2 & x_3\\
x_4 & x_5 & x_6 & x_7\\
x_8 & x_9 & x_{10} & x_{11}\\
x_{12} & x_{13} & x_{14} & x_{15}
\end{bmatrix}
\mapsto
\begin{bmatrix}
x_0 & x_1 & x_7 & x_6\\
x_4 & x_5 & x_3 & x_2\\
x_{13} & x_{12} & x_{10} & x_{11}\\
x_9 & x_8 & x_{14} & x_{15}
\end{bmatrix}.
\]
We are grateful to the authors of \cite{MZLLCMCQY25} for pointing out this additional generator to us after we mistakenly claimed that $G=G_0$. The group $G$ was first computed by Jordan \cite{J69} as the automorphism group of the Kummer configuration, which is the set of crosses, and the additional generator is the permutation $B$ in his paper; see also \cite{AS70} for a modern description of $G$. 

We can describe the odd $16$-dimensional kissing configuration as follows using $4320$ vectors of squared norm~$8$ in $\R^{16}$, which we view as $\R^{4 \times 4}$ as above:
\begin{enumerate}
\item[(a)] The $\binom{16}{2} \cdot 2^2 = 480$ vectors with two $\pm2$'s and $0$'s elsewhere.

\item[(b)] The $(12+18) \cdot 2^7 = 3840$ vectors with entries $\pm1$ and $0$, where the support is given by the pair and square vectors in $C_6$ and the number of minus signs is odd.
\end{enumerate}
It is a valid kissing configuration because any two vectors of type~(b) with distinct supports have at most four nonzero coordinates in common (the square and pair vectors have minimal distance $8$ between them), while those with the same support differ in at least two signs.

It is not hard to extend this kissing configuration to $5346$ vectors in $\R^{17}$, which we view as $\R^{4 \times 4} \times \R$:
\begin{enumerate}
\item[(1)] The $4320$ odd kissing vectors in $\R^{4 \times 4} \times \{0\}$, split into $480$ of type~(1a) and $3840$ of type~(1b) as above.

\item[(2)] The $4^2 \cdot 2^5 \cdot 2 = 1024$ vectors $(v,\pm \sqrt{2})$, where $v \in \R^{4 \times 4}$ has entries $\pm 1$ and $0$, with support given by the cross vectors in $C_6$ and an odd number of minus signs in $v$.

\item[(3)] The $2$ vectors $(0,\pm \sqrt{8})$, with $0$ being the zero vector in $\R^{4 \times 4}$.
\end{enumerate}
This configuration matches the size of Leech's construction in $17$ dimensions, and a short case analysis shows that it is a valid kissing configuration. Most cases are immediate: inner products between distinct vectors of type~(1) have already been handled, those between types~(2) and~(1a) are at most $2+2=4$, those between types~(2) and~(3) are at most $\sqrt{2} \cdot \sqrt{8} = 4$, those between types~(3) and~(1) are $0$, and those between distinct vectors of type~(3) are $-8$.

The more subtle cases are interactions between vectors of type~(2) or between those of types~(2) and~(1b), which are handled as follows. Two vectors of type~(2) with different cross vectors have inner product at most~$4$ because their cross vectors intersect in at most two coordinates, those with the same cross vector but opposite last coordinates have inner product at most $6 - 2 = 4$, and those with the same cross vector and the same last coordinate must differ in at least two signs and therefore have inner product at most $4 - 2 + 2 = 4$. Vectors of types~(2) and~(1b) have inner product at most~$4$ because codewords in $C_6$ of weights $6$ and $8$ must overlap in at most four coordinates (otherwise they could differ in at most $(6-5)+(8-5)=4$ coordinates, but the minimal distance of $C_6$ is $6$). This completes the case analysis.

For comparison, Leech's construction can be obtained by using even numbers of minus signs in the first sixteen coordinates, rather than odd. The even $17$-dimensional configuration is maximal, in the sense that there is no room to enlarge it. To check that the deep holes are not large enough to add any points, one can compute the facets of the convex hull of the $5346$ points using \cite{D} (see also \cite{DD16,D22}). Each facet corresponds to a hole, and checking all the facets shows that no holes are deep enough for additional points.

By contrast, the odd configuration can be enlarged, using points of the form
\[
((-1)^{c_1}2/3, \dots, (-1)^{c_{16}} 2/3, \pm \sqrt{8}/3),
\]
where the last sign is arbitrary and $(c_1,\dots,c_{16}) \in C_{10}$. In fact, one can check using~\cite{D} that these points are exactly the deep holes in the configuration, although our construction does not depend on this assertion.
To verify that these points have inner product at most~$4$ with the points so far, we simply use the fact that $C_{10}$ is orthogonal to $C_6$ modulo~$2$: for the vectors of types (1a), (1b), (2), and (3) above, we get inner products of at most
\[
\frac{2}{3} \cdot 2 \cdot 2 = \frac{8}{3} < 4,
\]
\[
\frac{2}{3} \cdot 1 \cdot 7 - \frac{2}{3} \cdot 1 \cdot 1 = 4,
\]
\[
\frac{2}{3} \cdot 1 \cdot 5 - \frac{2}{3} \cdot 1 \cdot 1 + \sqrt{2} \cdot \frac{\sqrt{8}}{3} = 4,
\]
and
\[
\sqrt{8} \cdot \frac{\sqrt{8}}{3} = \frac{8}{3} < 4,
\]
respectively.

Not all of these new vectors can be used simultaneously. For notational simplicity, we will index the vector
\[
((-1)^{c_1}2/3, \dots, (-1)^{c_{16}} 2/3, \pm \sqrt{8}/3),
\]
by $(c,\pm 1)$. Then the inner product of the vectors indexed by $(c,s)$ and $(c',s')$ is $(64 - 8 d(c,c') + 8ss')/9$, and the compatibility relationships between these potential additional points are as follows: 
\begin{enumerate}
    \item The vectors indexed by $(c,1)$ and $(c', -1)$ are too close if and only if $c = c'$, because $C_{10}$ has minimal distance~$4$. 
    \item The vectors indexed by $(c,s)$ and $(c', s)$ are too close if and only if $d(c,c') \le 4$.
\end{enumerate}
Thus, the maximum possible number of these candidate vectors that can be added to this configuration is the sum of the sizes of two disjoint subcodes (not necessarily linear) of $C_{10}$ that each have minimal distance greater than~$4$ (equivalently, at least~$6$, since $C_{10}$ is an even code). Lemma~\ref{lemma:192} gives an increase of $192 \cdot 2 = 384$, to achieve 
\[5346+384=5730
\]
in total. This proves Theorem~\ref{thm:main} for $17$ dimensions.

\begin{lemma} \label{lemma:192}
The largest possible size for a subcode of $C_{10}$ with minimal distance~$6$ is $192$, and in fact $C_{10}$ contains two disjoint such codes.
\end{lemma}

In the proof of this lemma, we will make several assertions about orbits under the action of $G$. The supplementary data \cite{CL} provides code for verifying these claims.

\begin{proof}
First, we prove an upper bound of $192$ using the linear programming bound for binary codes. Let $C$ be any nonempty subcode of $C_{10}$.
For each orbit $i$ of $G$ acting on $C_{10}$ (these orbits are enumerated in Table~\ref{table:orbits}), let 
\[
A_i = |\{(x,y) \in C^2 : \text{$x-y$ is in orbit $i$}\}|/|C|. 
\]
In other words, $A_i$ is the average over $x \in C$ of the number of $y \in C$ with $x-y$ in orbit~$i$.
Thus, we obtain a vector $(A_1,\dots,A_{10})$ with $A_i \ge 0$ for all $i$, $A_1=1$, and $A_1+\dots+A_{10} = |C|$. We call this vector the \emph{orbit enumerator} of $C$.

\begin{table}
\caption{The orbits of $G$ acting on $C_{10}$. Each entry lists the orbit number, the Hamming weight of the matrices in the orbit, the size of the orbit, and an orbit representative matrix.}
\label{table:orbits}
\centering
\begin{tabular}{cccccccc}
\toprule
orbit & weight & size & matrix & orbit &  weight & size & matrix\\
\cmidrule(r){1-4} \cmidrule(lr){5-8}
$1$ & $0$ & $1$ & $\begin{bsmallmatrix}0&0&0&0\\0&0&0&0\\0&0&0&0\\0&0&0&0\end{bsmallmatrix}$
 &
$6$ & $8$ & $360$ & $\begin{bsmallmatrix}1&1&0&0\\1&0&1&0\\0&1&0&1\\0&0&1&1\end{bsmallmatrix}$
\\
\addlinespace
$2$ & $4$ & $60$ & $\begin{bsmallmatrix}1&1&0&0\\1&1&0&0\\0&0&0&0\\0&0&0&0\end{bsmallmatrix}$
 &
$7$ &$10$ & $16$ & $\begin{bsmallmatrix}1&0&0&0\\0&1&1&1\\0&1&1&1\\0&1&1&1\end{bsmallmatrix}$
\\
\addlinespace
$3$ & $6$ & $16$ & $\begin{bsmallmatrix}0&1&1&1\\1&0&0&0\\1&0&0&0\\1&0&0&0\end{bsmallmatrix}$
 &
$8$ & $10$ & $240$ & $\begin{bsmallmatrix}0&0&1&1\\0&1&0&1\\1&0&0&1\\1&1&1&1\end{bsmallmatrix}$
\\
\addlinespace
$4$ & $6$ & $240$ & $\begin{bsmallmatrix}1&1&0&0\\1&0&1&0\\0&1&1&0\\0&0&0&0\end{bsmallmatrix}$
 &
$9$ & $12$ & $60$ & $\begin{bsmallmatrix}0&0&1&1\\0&0&1&1\\1&1&1&1\\1&1&1&1\end{bsmallmatrix}$
\\
\addlinespace
$5$ & $8$ & $30$ & $\begin{bsmallmatrix}1&1&0&0\\1&1&0&0\\0&0&1&1\\0&0&1&1\end{bsmallmatrix}$
 &
$10$ & $16$ & $1$ & $\begin{bsmallmatrix}1&1&1&1\\1&1&1&1\\1&1&1&1\\1&1&1&1\end{bsmallmatrix}$
\\
\bottomrule
\end{tabular}
\end{table}

A \emph{positive semidefinite kernel on $C_{10}$} is a function $K \colon C_{10} \times C_{10} \to \R$ that defines a $2^{10} \times 2^{10}$ positive semidefinite matrix indexed by $C_{10}$, and we say $K$ is \emph{invariant} if $K(x,y)$ depends only on the $G$-orbit of $x-y$. In that case, we write $K(i)$ for the value of $K(x,y)$ when $x-y$ is in orbit~$i$, and we write $\mathcal{K}$ for the $2^{10} \times 2^{10}$ matrix with the $(x,y)$ entry being $K(x,y)$.

The characters of the abelian group $C_{10}$ are the homomorphisms from $C_{10}$ to $\{\pm 1\}$. If $\chi$ is any character of $C_{10}$, then $(x,y) \mapsto \chi(x) \chi(y)$ is a positive semidefinite kernel of rank~$1$, and all translation-invariant positive semidefinite kernels are nonnegative linear combinations of these kernels by Fourier diagonalization, while $G$-invariance forces the coefficients to be constant on $G$-orbits of characters. Thus, if $\chi$ is any character, then
\[
K_\chi(x,y) = \frac{1}{|G|}\sum_{g \in G} \chi\big(g(x-y)\big) = \frac{1}{|G|}\sum_{g \in G} \chi(gx) \chi(gy)
\]
defines an invariant positive semidefinite kernel $K_\chi$, and these kernels generate the cone of invariant positive semidefinite kernels.

There are $10$ such kernels, corresponding to the orbits of $G$ on the group of characters of $C_{10}$. Specifically, for each $u \in \F_2^{4 \times 4}$, there is a character $\chi_u$ defined by $\chi_u(v) = (-1)^{\langle u,v \rangle}$. Since $C_6 = C_{10}^\perp$, the group of characters is isomorphic to $\F_2^{4 \times 4}/C_6$. The $G$-orbits of characters have the following representatives:
\[
\begin{split}
&\begin{bsmallmatrix}0&0&0&0\\0&0&0&0\\0&0&0&0\\0&0&0&0\end{bsmallmatrix},\ 
\begin{bsmallmatrix}0&0&0&0\\0&0&0&0\\0&0&0&0\\0&0&0&1\end{bsmallmatrix},\ 
\begin{bsmallmatrix}0&0&0&0\\0&0&0&0\\0&0&0&0\\0&0&1&1\end{bsmallmatrix},\ 
\begin{bsmallmatrix}0&0&0&0\\0&0&0&0\\0&0&0&0\\0&1&1&1\end{bsmallmatrix},\ 
\begin{bsmallmatrix}0&0&0&0\\0&0&0&0\\0&0&0&0\\1&1&1&1\end{bsmallmatrix},\\
&\begin{bsmallmatrix}0&0&0&0\\0&0&0&0\\0&0&0&1\\0&0&1&1\end{bsmallmatrix},\
\begin{bsmallmatrix}0&0&0&0\\0&0&0&0\\0&0&0&1\\0&1&1&1\end{bsmallmatrix},\
\begin{bsmallmatrix}0&0&0&0\\0&0&0&0\\0&0&1&1\\0&0&1&1\end{bsmallmatrix},\
\begin{bsmallmatrix}0&0&0&0\\0&0&0&1\\0&0&1&0\\0&1&1&1\end{bsmallmatrix},\
\begin{bsmallmatrix}0&0&0&0\\0&0&1&1\\0&1&0&1\\1&0&0&1\end{bsmallmatrix}.
\end{split}
\]

If $\chi = \chi_u$ and $v \in C_{10}$ is a representative of orbit~$i$ for the action of $G$ on $C_{10}$, then
\[
K_\chi(i) = \frac{1}{|G|} \sum_{g \in G} (-1)^{\langle gu, v \rangle}.
\]
Table~\ref{table:psdk} lists these values, with each row corresponding to one of the character orbit representatives listed above, in order.

\begin{table}
\caption{The invariant positive semidefinite kernels $K_\chi$ on $C_{10}$. Each row describes a kernel $K_\chi \colon C_{10} \times C_{10} \to \R$ by listing the value of $K_\chi(x,y)$ when $x-y$ is in each of the $10$ orbits from Table~\ref{table:orbits}, ordered as in that table.}
\label{table:psdk}
\centering
\begin{tabular}{cccccccccc}
\toprule \addlinespace
$1$ & $1$ & $1$ & $1$ & $1$ & $1$ & $1$ & $1$ & $1$ & $1$\\ \addlinespace
$1$ & $\tfrac{1}{2}$ & $\tfrac{1}{4}$ & $\tfrac{1}{4}$ & $0$ & $0$ & $-\tfrac{1}{4}$ & $-\tfrac{1}{4}$ & $-\tfrac{1}{2}$ & $-1$\\ \addlinespace
$1$ & $\tfrac{1}{5}$ & $0$ & $0$ & $-\tfrac{1}{15}$ & $-\tfrac{1}{15}$ & $0$ & $0$ & $\tfrac{1}{5}$ & $1$\\ \addlinespace
$1$ & $\tfrac{1}{10}$ & $-\tfrac{1}{4}$ & $-\tfrac{1}{20}$ & $0$ & $0$ & $\tfrac{1}{4}$ & $\tfrac{1}{20}$ & $-\tfrac{1}{10}$ & $-1$\\ \addlinespace
$1$ & $\tfrac{1}{5}$ & $-1$ & $-\tfrac{1}{5}$ & $1$ & $\tfrac{1}{5}$ & $-1$ & $-\tfrac{1}{5}$ & $\tfrac{1}{5}$ & $1$\\ \addlinespace
$1$ & $-\tfrac{1}{30}$ & $\tfrac{1}{4}$ & $-\tfrac{1}{60}$ & $0$ & $0$ & $-\tfrac{1}{4}$ & $\tfrac{1}{60}$ & $\tfrac{1}{30}$ & $-1$\\ \addlinespace
$1$ & $-\tfrac{1}{15}$ & $0$ & $0$ & $-\tfrac{1}{15}$ & $\tfrac{1}{45}$ & $0$ & $0$ & $-\tfrac{1}{15}$ & $1$\\ \addlinespace
$1$ & $-\tfrac{1}{15}$ & $1$ & $-\tfrac{1}{15}$ & $1$ & $-\tfrac{1}{15}$ & $1$ & $-\tfrac{1}{15}$ & $-\tfrac{1}{15}$ & $1$\\ \addlinespace
$1$ & $-\tfrac{1}{6}$ & $-\tfrac{1}{4}$ & $\tfrac{1}{12}$ & $0$ & $0$ & $\tfrac{1}{4}$ & $-\tfrac{1}{12}$ & $\tfrac{1}{6}$ & $-1$\\ \addlinespace
$1$ & $-\tfrac{1}{3}$ & $-1$ & $\tfrac{1}{3}$ & $1$ & $-\tfrac{1}{3}$ & $-1$ & $\tfrac{1}{3}$ & $-\tfrac{1}{3}$ & $1$\\ \addlinespace \bottomrule
\end{tabular}
\end{table}

If $K$ is an invariant positive semidefinite kernel and $C$ is a subcode of $C_{10}$ with orbit enumerator $(A_1,\dots,A_{10})$, then
\[
\sum_{i=1}^{10} K(i) A_i = \frac{1}{|C|} \sum_{x,y \in C} K(x,y) = \frac{1}{|C|} v_C^\top \mathcal{K} v_C \ge 0,
\]
where $v_C$ is the characteristic function of $C$, i.e., the vector indexed by $C_{10}$ with a $1$ for each element of $C$.
Thus, the orbit enumerator of a code $C$ with minimal distance~$6$ must satisfy the following constraints:
\begin{enumerate}
\item $A_1 = 1$,

\item $A_2 = 0$ (since that is the orbit with Hamming weight~$4$),

\item $A_i \ge 0$ for all $i$, and

\item $\sum_{i=1}^{10} K(i) A_i \ge 0$ for each invariant positive semidefinite kernel $K$ listed in Table~\ref{table:psdk}.
\end{enumerate}

Subject to these linear constraints, we wish to maximize $|C| = \sum_{i=1}^{10} A_i$, which is a linear function of the orbit enumerator.
Solving this linear program shows that $|C| \le 192$, and equality can be achieved only if
\begin{equation} \label{eq:A}
A = (1, 0, 0, 80, 30, 0, 0, 80, 0, 1).
\end{equation}
To see why, we can use a solution of the dual linear program; specifically, we will obtain this bound by taking a linear combination of the inequalities $\sum_{i=1}^{10} K(i) A_i \ge 0$ with nonnegative coefficients that are chosen to optimize the resulting bound.
Let $K_1,\dots,K_{10}$ be the kernels listed in Table~\ref{table:psdk}, and let the corresponding coefficients be
\[
c = \left(1, 13, 57, 56, 5, 27, 33, 0, 0, 0\right).
\]
Then the inequalities $\sum_{i=1}^{10} K_j(i) A_i \ge 0$ imply that
\begin{equation} \label{eq:LPdualineq}
\sum_{i=1}^{10} A_i = c_1 \sum_{i=1}^{10} K_1(i) A_i \le \sum_{j=1}^{10} c_j \sum_{i=1}^{10} K_j(i) A_i,
\end{equation}
and the right side of this inequality is
\[
192 A_1 + \frac{112}{5} A_2 - 8 A_3 - \frac{16}{15}A_6.
\]
Because $A_1=1$, $A_2=0$, and $A_i \ge 0$ for $i \ge 3$, we obtain an upper bound of $192$, with equality only when $A_3=A_6=0$. Furthermore, by complementary slackness \eqref{eq:LPdualineq} can be sharp only if $\sum_{i=1}^{10} K_j(i) A_i =0$ whenever $j>1$ and $c_j>0$ (in other words, for all $j$ from $2$ to $7$). Solving the linear equations from the previous two sentences for $A_1,\dots,A_{10}$ yields the unique solution \eqref{eq:A}.

This solution gives an important clue to construct a code $C$ achieving equality: $A_5=30$, which is the full size of orbit~$5$. In other words, the average number of neighbors in orbit~$5$ is the same as the maximum number, which means $C$ must be closed under adding elements of orbit~$5$. In the terminology we used for $C_6$, this orbit consists of pairs and squares. They generate a linear code $C_5$ of dimension~$5$, which also contains the zero vector and the all $1$'s vector (and is the Reed-Muller code $\RM(1,4)$). Thus, an optimal code $C$ must consist of six cosets of $C_5$.

There are $2^{10-5}=32$ cosets of $C_5$ in $C_{10}$. Each pair of distinct cosets has minimal distance~$4$ or~$6$ between them, and we want to find six cosets that are all at minimal distance~$6$ from each other. The quotient $C_{10}/C_5$ is spanned by the cosets of the five matrices
\[
x_1=\begin{bsmallmatrix}0&0&0&0\\0&0&0&0\\0&0&1&1\\0&0&1&1\end{bsmallmatrix},\ 
x_2=\begin{bsmallmatrix}0&0&0&0\\0&0&0&0\\0&1&0&1\\0&1&0&1\end{bsmallmatrix},\ 
x_3=\begin{bsmallmatrix}0&0&0&0\\0&0&1&1\\0&0&0&0\\0&0&1&1\end{bsmallmatrix},\ 
x_4=\begin{bsmallmatrix}0&0&0&0\\0&1&0&1\\0&0&0&0\\0&1&0&1\end{bsmallmatrix},\ 
x_5=\begin{bsmallmatrix}0&0&0&1\\0&0&0&1\\0&0&0&1\\1&1&1&0\end{bsmallmatrix}.
\]
Define a linear isomorphism $f \colon C_{10}/C_5 \to \{v \in \F_2^6 : v_1+\dots+v_6 = 0\}$ by
\begin{align*}
f(x_1+C_5) &= (1,1,1,1,0,0),\\
f(x_2+C_5) &= (1,1,0,0,1,1),\\
f(x_3+C_5) &= (0,1,1,0,1,1),\\
f(x_4+C_5) &= (1,0,1,1,1,0),\ \text{and}\\
f(x_5+C_5) &= (1,1,1,1,1,1).
\end{align*}
Then one can check that two cosets have minimal distance~$4$ between them if and only if their images under $f$ have Hamming distance~$4$ between them. Now the question becomes which subsets of six points in $\{v \in \F_2^6 : v_1+\dots+v_6 = 0\}$ avoid Hamming distance~$4$. Such subsets are easy to describe, although the proof does not require classifying them: they are all obtained by translating and permuting the coordinates of the points $(0,0,0,0,0,0)$, $(1,1,0,0,0,0)$, $(1,0,1,0,0,0)$, $(1,0,0,1,0,0)$, $(1,0,0,0,1,0)$, and $(1,0,0,0,0,1)$. In particular, this set and its translation by $(1,1,1,1,1,1)$ are disjoint, because the translated points have weights~$6$ or~$4$ while the original points have weights~$0$ or~$2$. The preimages of these two sets under $f$ give two disjoint sets of six cosets of $C_5$ in $C_{10}$, with pairwise minimal distance~$6$ between the cosets within each set, as desired.
\end{proof}

If we normalize the vectors in the kissing configuration to be unit vectors, then the inner products that occur among them are
\[
-1,\ -\frac{7}{9},\ -\frac{5}{9},\ -\frac{1}{2},\ -\frac{1}{3},\ -\frac{1}{4},\ -\frac{1}{6},\ -\frac{1}{9},\ 0,\ \frac{1}{9},\ \frac{1}{6},\ \frac{1}{4},\ \frac{1}{3},\ \frac{1}{2},\ 1.
\]
It follows that this configuration cannot be a cross section of the Leech lattice minimal vectors.

\section{Dimension 18}
\label{sec:18}

The case of $18$ dimensions is similar in spirit, but somewhat more complicated. We identify $\R^{18}$ with $\R^{4 \times 4} \times \R \times \R$. First, we will match the size of Leech's construction by producing $7398$ vectors using an odd number of minus signs. We begin with the original $5346$ odd kissing vectors in $\R^{4 \times 4} \times \R \times \{0\}$. Note that we do not include the additional $384$ vectors, although we will use some of them later. Next, we include the four vectors $(0,\pm \sqrt{2}, \pm \sqrt{6})$, with $0$ being the zero vector in $\R^{4 \times 4}$. Finally, we include $2048$ vectors constructed as follows. Let $A$ be the set of $4^2 \cdot 2^5$ matrices given by independently cyclically shifting the rows and columns of
\[
\begin{bmatrix}
0 & 0 & 0 & 0\\
1 & 1 & 0 & 0\\
1 & 0 & 1 & 0\\
1 & 0 & 0 & 1
\end{bmatrix}
\]
and changing an odd number of $1$'s to $-1$,
and let $B$ be the analogous set starting with
\[
\begin{bmatrix}
1 & 1 & 1 & 0\\
0 & 0 & 1 & 0\\
0 & 1 & 0 & 0\\
1 & 0 & 0 & 0
\end{bmatrix}.
\]
Then we include the vectors $(a,\sqrt{2}/2, \sqrt{6}/2)$ or $(a,-\sqrt{2}/2,-\sqrt{6}/2)$, where $a$ is an element of $A$, and the vectors $(b,\sqrt{2}/2, -\sqrt{6}/2)$ or $(b,-\sqrt{2}/2,\sqrt{6}/2)$, where $b$ is an element of $B$. The following case analysis shows that the resulting configuration is in fact a valid kissing configuration of size $7398$.

The $5346$ odd kissing vectors have already been checked in the previous section, their inner products with $(0,\pm \sqrt{2}, \pm \sqrt{6})$ are at most $\sqrt{8} \cdot \sqrt{2} = 4$, and the inner products among $(0,\pm \sqrt{2}, \pm \sqrt{6})$ are at most $6-2 = 4$. The remaining cases are as follows, with $v$ and $w$ distinct vectors in the configuration:
\begin{enumerate}
\item Suppose $v$ and $w$ either both have type $A$ or both have type $B$. If they have distinct supports in $\R^{4 \times 4}$, then these supports intersect in two coordinates, and the last two coordinates have inner product at most $1/2+3/2=2$, so $\langle v,w \rangle \le 2+2=4$. If they have the same support in $\R^{4 \times 4}$, then either the signs differ in at least two of the first sixteen coordinates, with $\langle v,w \rangle \le 4-2+2=4$, or the final two coordinates are opposite, with $\langle v,w \rangle \le 6 - 2 = 4$.

\item If $v$ has type $A$ and $w$ has type $B$, then their supports in $\R^{4 \times 4}$ intersect in at most three coordinates, and the final two coordinates have inner product at most $3/2-1/2=1$, so $\langle v,w \rangle \le 3+1 = 4$.

\item If $v$ has type $A$ or $B$ and $w$ has a cross support (type~(2) in $17$ dimensions), then their supports in $\R^{4 \times 4}$ intersect in at most three coordinates, and the seventeenth coordinate product is at most $1$, so $\langle v,w \rangle \le 3+1 = 4$.

\item If $v$ has type $A$ or $B$ and $w$ has a pair or square support (type~(1b) in $17$ dimensions), then their supports in $\R^{4 \times 4}$ intersect in at most four coordinates, and there is no contribution from the last two coordinates, so $\langle v,w \rangle \le 4$.

\item If $v$ has type $A$ or $B$ and $w$ has two $\pm 2$ coordinates (type~(1a) in $17$ dimensions), then $\langle v,w \rangle \le 2+2 = 4$.

\item Finally, if $v$ has type $A$ or $B$ and $w \in \{(0,\pm \sqrt{8},0), (0, \pm \sqrt{2}, \pm \sqrt{6})\}$, then $\langle v,w \rangle \le 4$.
\end{enumerate}

To achieve a kissing configuration of size $7654$, we will include a subset of $256$ additional points, each of the form
\[
((-1)^{c_1}2/3, \dots, (-1)^{c_{16}} 2/3, \pm \sqrt{8}/3, 0).
\]
In other words, they are among the deep holes from the odd $17$-dimensional configuration. However, we cannot use as many of them as before, because they interfere with the other $18$-dimensional vectors. Instead, we will construct them using a code $C_8$.

Let $C_8 \subseteq \F_2^{4 \times 4}$ be the code consisting of matrices for which the rows, columns, main diagonal, and first shifted diagonal all have the same sum in $\F_2$. Here the \emph{main diagonal} of a matrix in $\F_2^{4 \times 4}$ is the support of
\[ \begin{bmatrix}
    1 & 0 & 0 & 0 \\ 0 & 1 & 0 & 0 \\ 0 & 0 & 1 & 0 \\ 0 & 0 & 0 & 1
\end{bmatrix},\]
and the \emph{first shifted diagonal} is the support of
\[ \begin{bmatrix}
    0 & 1 & 0 & 0 \\ 0 & 0 & 1 & 0 \\ 0 & 0 & 0 & 1 \\ 1 & 0 & 0 & 0
\end{bmatrix}.\]
Then $C_8$ is an $8$-dimensional linear subcode of $C_{10}$. Indeed, it has codimension~$2$ in $C_{10}$ because the matrices
\[
\begin{bmatrix}
    1 & 1 & 0 & 0 \\ 1 & 1 & 0 & 0 \\ 0 & 0 & 0 & 0 \\ 0 & 0 & 0 & 0
\end{bmatrix}
\quad\text{and}\quad
\begin{bmatrix}
    1 & 0 & 0 & 1 \\ 1 & 0 & 0 & 1 \\ 0 & 0 & 0 & 0 \\ 0 & 0 & 0 & 0
\end{bmatrix}
\]
have zero row and column sums, while their main and first shifted diagonal sums are $(0,1)$ and $(1,0)$, respectively.

The row and column constraints defining $C_8$ are preserved under a cyclic shift of the rows or columns, while the main diagonal and first shifted diagonal can transform into the other two shifted diagonals. However, these further shifted diagonals impose no new constraints, because they are congruent modulo $C_8^\perp$ to the main diagonal and first shifted diagonal:
\[
\begin{bmatrix} 0&0&1&0\\0&0&0&1\\1&0&0&0\\0&1&0&0\end{bmatrix}
= \begin{bmatrix} 1&1&1&1\\0&0&0&0\\1&1&1&1\\0&0&0&0\end{bmatrix}
+ \begin{bmatrix} 0&1&0&1\\0&1&0&1\\0&1&0&1\\0&1&0&1\end{bmatrix}
+ \begin{bmatrix} 1&0&0&0\\0&1&0&0\\0&0&1&0\\0&0&0&1\end{bmatrix}
\]
and
\[
\begin{bmatrix} 0&0&0&1\\1&0&0&0\\0&1&0&0\\0&0&1&0\end{bmatrix}
= \begin{bmatrix} 1&1&1&1\\0&0&0&0\\1&1&1&1\\0&0&0&0\end{bmatrix}
+ \begin{bmatrix} 1&0&1&0\\1&0&1&0\\1&0&1&0\\1&0&1&0\end{bmatrix}
+ \begin{bmatrix} 0&1&0&0\\0&0&1&0\\0&0&0&1\\1&0&0&0\end{bmatrix}.
\]
Thus, $C_8$ is closed under cyclic shifts of the rows or columns. 

For each codeword $c \in C_8$, we include the vector
\[
((-1)^{c_1}2/3, \dots, (-1)^{c_{16}} 2/3, (-1)^{\langle c,e\rangle} \sqrt{8}/3, 0)
\]
in the kissing configuration, where
\[
e = \begin{bmatrix}
    1 & 1 & 0 & 0 \\ 1 & 1 & 0 & 0 \\ 0 & 0 & 0 & 0 \\ 0 & 0 & 0 & 0
\end{bmatrix}.
\]
Checking that this construction works requires some case analysis, based on the following lemma.

\begin{lemma} \label{lem:18}
Every support occurring in $A\cup B$ lies in $C_8^\perp$, and every nonzero $x \in C_8$ has weight at least~$6$ if $\langle x,e \rangle = 0$ and weight at least~$4$ if $\langle x,e \rangle = 1$.
\end{lemma}

\begin{proof}
Because $C_8$ is closed under cyclic shifts, we can show that the supports in $A \cup B$ lie in $C_8^\perp$ by checking that
\[
\begin{bmatrix}
0 & 0 & 0 & 0\\
1 & 1 & 0 & 0\\
1 & 0 & 1 & 0\\
1 & 0 & 0 & 1
\end{bmatrix} =
\begin{bmatrix}
1 & 0 & 0 & 0\\
1 & 0 & 0 & 0\\
1 & 0 & 0 & 0\\
1 & 0 & 0 & 0
\end{bmatrix}
+
\begin{bmatrix}
1 & 0 & 0 & 0\\
0 & 1 & 0 & 0\\
0 & 0 & 1 & 0\\
0 & 0 & 0 & 1
\end{bmatrix}
\]
and
\[
\begin{bmatrix}
1 & 1 & 1 & 0\\
0 & 0 & 1 & 0\\
0 & 1 & 0 & 0\\
1 & 0 & 0 & 0
\end{bmatrix}
= \begin{bmatrix}
0 & 0 & 0 & 0\\
1 & 1 & 1 & 1\\
0 & 0 & 0 & 0\\
0 & 0 & 0 & 0
\end{bmatrix}
+ \begin{bmatrix}
1 & 1 & 0 & 0\\
1 & 1 & 0 & 0\\
1 & 1 & 0 & 0\\
1 & 1 & 0 & 0
\end{bmatrix}
+ \begin{bmatrix}
0 & 0 & 1& 0\\
0 & 0 & 0 & 1\\
1 & 0 & 0 & 0\\
0 & 1 & 0 & 0
\end{bmatrix},
\]
since the sum of any two row, column, or shifted diagonal indicator matrices is in $C_8^\perp$.

Every codeword in $C_8$ has even weight, and there is no word of weight $2$. Indeed, if the row and column sums are odd, then the weight must be at least $4$, while if they are even, then two nonzero entries would have to lie in the same row and the same column.

All that remains is to show that the codewords $x$ of weight $4$ must satisfy $\langle x,e \rangle = 1$. If the row and column sums are odd, then $x$ must be a permutation matrix. Furthermore, it must intersect each of the four shifts of the main diagonal in one entry. If the nonzero entry in row $i$ is in column $\pi(i)$, then $\{\pi(i)-i : i \in \Z/4\Z\} = \Z/4\Z$ and therefore
\[
0 \equiv \sum_{i=1}^4 (\pi(i)-i) \equiv \sum_{j=1}^4 j \equiv 2 \pmod{4},
\]
which is a contradiction.

Thus, the row and column sums must be even. The support of $x$ must therefore be a combinatorial rectangle formed by two rows and two columns. Up to cyclic shifts of the rows and columns, there are four possibilities: the rows may or may not be adjacent, and the columns may or may not be adjacent. If either is adjacent, then the diagonal conditions are violated:
\begin{align*}
\left\langle\begin{bmatrix}
1 & 1 & 0 & 0\\
0 & 0 & 0 & 0\\
1 & 1 & 0 & 0\\
0 & 0 & 0 & 0
\end{bmatrix},
\begin{bmatrix}
1 & 0 & 0 & 0\\
0 & 1 & 0 & 0\\
0 & 0 & 1 & 0\\
0 & 0 & 0 & 1
\end{bmatrix}
\right\rangle &= 
\left\langle\begin{bmatrix}
1 & 0 & 1 & 0\\
1 & 0 & 1 & 0\\
0 & 0 & 0 & 0\\
0 & 0 & 0 & 0
\end{bmatrix},
\begin{bmatrix}
1 & 0 & 0 & 0\\
0 & 1 & 0 & 0\\
0 & 0 & 1 & 0\\
0 & 0 & 0 & 1
\end{bmatrix}
\right\rangle\\
& =
\left\langle\begin{bmatrix}
1 & 1 & 0 & 0\\
1 & 1 & 0 & 0\\
0 & 0 & 0 & 0\\
0 & 0 & 0 & 0
\end{bmatrix},
\begin{bmatrix}
0 & 1 & 0 & 0\\
0 & 0 & 1 &0\\
0 & 0 & 0 & 1\\
1 & 0 & 0 & 0
\end{bmatrix}
\right\rangle = 1.
\end{align*}
Thus, $x$ must be a shift of
\[
\begin{bmatrix}
1 & 0 & 1 & 0\\
0 & 0 & 0 & 0\\
1 & 0 & 1 & 0\\
0 & 0 & 0 & 0
\end{bmatrix},
\]
and any such shift has odd inner product with $e$, as desired.
\end{proof}

To complete the proof of Theorem~\ref{thm:main}, we must check that the points
\[
((-1)^{c_1}2/3, \dots, (-1)^{c_{16}} 2/3, (-1)^{\langle c,e \rangle}\sqrt{8}/3, 0)
\]
indexed by $c \in C_8$ are compatible with each other and the $7398$ points already in the kissing configuration. As in $17$ dimensions, they are compatible with each other because codewords $c,c' \in C_8$ with $c \ne c'$ satisfy $(-1)^{\langle c,e \rangle} \ne (-1)^{\langle c',e \rangle}$ when $d(c,c') = 4$, by Lemma~\ref{lem:18} and linearity. They are compatible with the $5346$ points from $17$ dimensions by our $17$-dimensional calculations, and they are compatible with $(0, \pm \sqrt{2}, \pm \sqrt{6})$ since $\sqrt{2} \cdot \sqrt{8}/3 = 4/3 < 4$. Finally, if $v$ is one of the $2048$ vectors of type $A$ or $B$ and $c \in C_8$, then the support of $v$ in $\R^{4 \times 4}$ is orthogonal to $c$ by Lemma~\ref{lem:18}. The point indexed by $c$ therefore has a sign difference from $v$ in the first sixteen coordinates, because their numbers of minus signs have opposite parities in the support of $v$, and their inner product is at most
\[
(5 - 1) \cdot \frac{2}{3} + \frac{\sqrt{2}}{2} \cdot \frac{\sqrt{8}}{3} = \frac{10}{3} < 4,
\]
as desired.

If we normalize these $7654$ vectors to be unit vectors, then their inner products include every inner product in the list from $17$ dimensions, together with $\pm1/12$ and $\pm 5/12$. Thus, again they are not a cross section of the Leech lattice minimal vectors.

\section{Potential generalizations}
\label{sec:gen}

Our construction in $17$ dimensions is appealing, but we see no reason to believe it is truly optimal. Optimality seems even less likely in $20$ and $21$ dimensions, and particularly unlikely in $18$ dimensions (we omit $19$ because Ho \cite{H26} has already identified an improved code for use in our $19$-dimensional construction). Furthermore, the optimal kissing configurations probably cannot be obtained by fine-tuning our construction. Instead, they will presumably require a different idea. From this perspective, the value of our work lies in showing that the laminated lattices are not optimal, rather than in producing realistic candidates for optimality.

It is unclear to us what the limits of these sorts of constructions might be. It might be possible to use the $21$-dimensional construction to improve on $22$ or even $23$ dimensions, much like we used the odd $16$-dimensional kissing configuration to improve on $17$ and $18$ dimensions. Another tantalizing possibility would be to improve the Euclidean sphere packing density in any of these dimensions.

Among all the cases discussed in this paper, the $16$-dimensional configurations may be the most likely to be optimal. At the very least, the $4320$-point Barnes-Wall kissing configuration seems difficult to beat. However, the fact that it is not unique shows that it is not as special as the kissing configurations in $4$, $8$, or $24$ dimensions, for which uniqueness is known \cite{dLLdMK24,BS81}. In some sense the odd $16$-dimensional kissing configuration is more capacious than the Barnes-Wall configuration, because it leaves room to add more points in $17$ dimensions. Perhaps the ability to make room in this way is a clue that the $16$-dimensional kissing number could be greater than $4320$.

\section*{Use of artificial intelligence}

Artificial intelligence played no role in these discoveries, but GPT-5.6~Sol helped proofread the paper and write code for the supplementary data.

\section*{Acknowledgments}

We are grateful to the referees for providing helpful suggestions, and Cohn thanks OpenAI for providing a ChatGPT Pro subscription at no cost. The bulk of this work was completed while Li was a summer intern at Microsoft Research New England.

\end{document}